\documentclass[oneside,11pt]{amsart}
\usepackage{amsmath, amsfonts,amsthm,times,graphics}
\usepackage[active]{srcltx}
 \makeatletter
\renewcommand*\subjclass[2][2000]{%
  \def\@subjclass{#2}%
  \@ifundefined{subjclassname@#1}{%
    \ClassWarning{\@classname}{Unknown edition (#1) of Mathematics
      Subject Classification; using '1991'.}%
  }{%
    \@xp\let\@xp\subjclassname\csname subjclassname@#1\endcsname
  }%
}
 \makeatother

\newtheorem{theorem}{Theorem}[section]

\newtheorem{lemma}[theorem]{Lemma}

\newtheorem{proposition}[theorem]{Proposition}
\theoremstyle{definition}

\newtheorem{remark}[theorem]{Remark}
\numberwithin{equation}{section}
\newcommand{\abs}[1]{\lvert#1\rvert}

\newcounter{minutes}
\divide\time by 60
\newcounter{hours}
\multiply\time by 60 \addtocounter{minutes}{-\time}
\begin{document}

\title[On some Riesz and Carleman type inequalities for harmonic functions ]{On some Riesz and Carleman type inequalities for harmonic functions on the unit disk}

\author{David Kalaj}
\address{Faculty of Natural Sciences and Mathematics, University of
Montenegro, Cetinjski put b.b. 81000 Podgorica, Montenegro}
\email{davidk@ac.me}

\author{Elver Bajrami}
\address{Department of Mathematics, University of Prishtina, Mother Teresa, No. 5, 10000, Prishtina, Kosovo}
\email{elver.bajrami@uni-pr.edu }
\def\thefootnote{}
\footnotetext{ \texttt{\tiny File:~\jobname.tex,
          printed: \number\year-\number\month-\number\day,
          \thehours.\ifnum\theminutes<10{0}\fi\theminutes }
} \makeatletter\def\thefootnote{\@arabic\c@footnote}\makeatother
%===============================================================================

\begin{abstract}
We prove some isoperimetric type inequalities for real harmonic functions in the unit disk belonging to the Hardy space $h^p$, $p>1$ and for complex harmonic functions in $h^4$. The results extend some recent results on the area. Further we discus some Riesz type results for holomorphic functions.
\end{abstract}
\maketitle

\section{Introduction and statement of main results}

Throughout the paper we let $\mathbf{U}=\{z\in\Bbb C:  |z|<1\}$ be the
open unit disk in the complex plane $\Bbb C$, and let $\mathbf{T}=\{z\in\Bbb C:  |z|=1\}$ be the unit circle in $\Bbb C$. The
normalized area measure on $\mathbf{U}$ will be denoted by $d\sigma$. In
terms of real (rectangular and polar) coordinates, we have
  $$
d\sigma=\frac{1}{\pi}dxdy=\frac{1}{\pi}rdrd\theta,\quad
z=x+iy=re^{i\theta}.
 $$
For $0<p<+\infty$ let $L^p(\mathbf{U},\sigma)=L^p$ denote the familiar
Lebesgue space on $\mathbf{U}$ with respect to the measure $\sigma$. The
{\it Bergman  space} $A^p(\mathbf{U})=A^p$ is the space of all
holomorphic functions in $L^p(\mathbf{U},\sigma)$. For a fixed $1\le
p<+\infty$, denote by $A^p_0$ the set of all functions $f\in A^p$
for which $f(0)=0$. %It is easy to show (cf. [HKZ, Section 3.2, p.
%55-56])  that $A^p_0$ is a closed subspace of $A^p$. In particular,
%$A^p_0$ is {\it invariant} subspace  of $A^p$; that is, $zf\in
%A^p_0$ whenever $f\in A^p_0$ (here $z$ denotes the identity function
%on $\mathbf{U}$).
For a function $f$ in $L^p(\mathbf{U},\sigma)$, we write
    $$
\Vert f\Vert_p=\left(\int_{\mathbf{U}}|f(z)|^p d\sigma\right)^{1/p}.
  $$

Bergman space $b^p$ of harmonic functions is defined similarly.

The {\it Hardy space} $h^p$ is defined as the space of (complex)
harmonic functions $f$ such that
$$\|f\|_{h^p}:=\sup_{0\le r < 1}(\int_{0}^{2\pi}|f(re^{it})|^p
\frac{dt}{2\pi})^{1/p}<\infty.$$

If $f\in h^p$, then by \cite[Theorem~6.13]{axl}, there exists
$$f(e^{it})=\lim_{r\to 1} f(re^{it}),   a.e.$$  and $f\in L^p(\mathbf{T}).$
It can be shown that there hold
$$\|f\|^p_{h^p}=\lim_{r\to 1}\int_{0}^{2\pi}|f(re^{it})|^p \frac{dt}{2\pi}
= \int_{0}^{2\pi}|f(e^{it})|^p \frac{dt}{2\pi}.$$

Similarly we define the {\it Hardy space} $H^p$ of holomorphic
functions.

The starting point of this paper is the well known isoperimetric
inequality for Jordan domains and isoperimetric inequality for
minimal surfaces due to Carleman \cite{tc}. In that paper Carleman,
among the other results proved that if $u$ is harmonic and smooth in
$\overline {\mathbf{U}}$ then $$\int_{\mathbf{U}}e^{2u}dx dy\le
\frac{1}{4\pi}(\int_0^{2\pi}e^u dt)^2.$$ By using a similar approach
as Carleman, Strebel (\cite{ks}) proved the isoperimetric inequality
for holomorphic functions; that is if $f\in H^1(\mathbf{U})$ then
\begin{equation}\label{isoper}\int_{\mathbf{U}}|f(z)|^2 dxdy \le \frac{1}{4\pi} (\int_{\mathbf{T}}|f(e^{it})|dt)^2.\end{equation} This inequality has been proved
independently by Mateljevi\'c and Pavlovi\'c (\cite{mp}). In
\cite{hw} have been done some generalizations for the space.

In this paper we prove the following results.

\begin{theorem}\label{cp}
Let $f$ be a real harmonic function. For $p>1$ we have that \begin{equation}\label{cep}\| f\|_{b^{2p}}\le C_p\| f\|_{h^{p}}\end{equation} where
$$C_p  = M_{2p} =\left\{
        \begin{array}{ll}
          \frac{\cos{\frac{\pi}{4p}}}{\cos\frac{\pi}{2p}}, & \hbox{if $1<p\le 2$;} \\
          \frac{\cos{\frac{\pi}{4p}}}{\sin\frac{\pi}{2p}}, & \hbox{if $p\ge 2$.} \\
        \end{array}
      \right.
$$
\end{theorem}

The inequality \eqref{cep} extends the main result in \cite{eb}, where Bajrami proved the same result but only for $p=4$.
But the wrong constant  appear in \cite{eb}, due to a wrong citation to the Duren approach \cite[p.~67-68]{duren}. However the same  method produces the same constant namely $C_4= \frac{1}{2\sin\frac{\pi}{16}}\approx 2.56292$. We were not able to check if \eqref{cep} is sharp. We are thankful to professor Hedenmalm who drawn attention to his paper \cite{hh}, where is treated a problem which suggests that the inequalities from Theorem~\ref{hed}, which we use in order to prove \eqref{cep},  are maybe not sharp.
The inequality \eqref{cep} improves similar results for real harmonic functions proved by Kalaj and Me\v strovi\' c in \cite{ka1} and by  Chen and Ponnusamy and Wang in \cite{spw}.
We expect that the inequality \eqref{cep} is true for complex harmonic mappings for every $p>1$. Kalaj and Me\v strovi\'c in \cite{ka1} proved it for $p=2$, namely they obtained that $C_2=\frac{1}{2\sin\frac{\pi}{8}}\approx 1.3$. On the same paper the example $f_a(z)=\Re \frac{z}{1-az}$, when $a\uparrow 1$ produces the constant $C_0=(5/2)^{1/4}\approx 1.257$, so the constant $C_2$ is not far from the sharp constant.  In this paper  we extend it for $p=4$.

Namely we have
\begin{theorem}\label{c4}
If $f\in h^4(\mathbf{U})$ is a nonzero complex  harmonic function then
$f\in b^8$ and there hold the inequality

$$\|f\|_{b^8}\le \frac{1}{2\sin\frac{\pi}{16}}\|f\|_{h^4}.$$

\end{theorem}
The main ingredient of the proofs are some sharp M. Riesz type inequalities proved for Hardy space by Verbitsky (\cite{ver}):
\begin{proposition}[Riesz type inequality]\label{verbi}
Let  $p> 1$. For every holomorphic mapping  there hold the sharp inequalities
$${L_p}\|\Re f\|_{h^p}\le \|f\|_{h^p}\le R_p\|\Re f\|_{h^p},$$ provided $\abs{\arg{f(0)}-\frac{\pi}{2}}\ge \frac{\pi }{2\bar p}$, or $f(0)=0$, where $\bar{p}=\max\{p,\frac{p}{p-1}\}$ and

$$R_p = \left\{
          \begin{array}{ll}
            \frac{1}{\cos\frac{\pi}{2p}}, & \hbox{if $1<p\le 2$;} \\
            \frac{1}{\sin\frac{\pi}{2p}}, & \hbox{if $p\ge 2$.}
          \end{array}
        \right.
\text{    and     } L_p = \left\{
          \begin{array}{ll}
            \frac{1}{\sin\frac{\pi}{2p}}, & \hbox{if $1<p\le 2$;} \\
            \frac{1}{\cos\frac{\pi}{2p}}, & \hbox{if $p\ge 2$.}
          \end{array}
        \right.
$$
\end{proposition}
In addition we refer to the references \cite{verb1} and \cite{pik} for related results.

By results of Verbitsky we prove some similar inequalities for Bergman space (Theorem~\ref{hed}).

From Proposition~\ref{verbi}, we obtain that, if $f=u+iv$ is holomorphic with $f(0)=0$, then
\begin{equation}\label{motiv} \|f\|^p_{H^p}\le \frac{R^p_p}{2} (\|u\|^p_{h^p}+\|v\|^p_{h^p}) \end{equation}
and
\begin{equation}\label{motiv1}\frac{L^p_p}{2} (\|u\|^p_{h^p}+\|v\|^p_{h^p})\le \|f\|^p_{H^p}. \end{equation}

Now  we formulate a similar result, where $\frac{R^p_p}{2} $ and $\frac{L^p_p}{2} $ are replaced by smaller, respectively bigger constants for $p$ close to $4$.

\begin{theorem}\label{newt}
Let $p\ge 2$ and let  $f=u+i v$ be a holomorphic function on the unit disk so that $f(0)=0$ and let $f\in H^p$. Then
for $p\in[2,4]$ we have
\begin{equation}\label{fini}\|f\|_{H^p}\le \left(\frac{p}{p-1}\right)^{1/p}  \left(\|u\|_{h^p}^p+\|v\|_{h^p}^p\right)^{1/p}\end{equation} and for
$p\ge 4$ we have

\begin{equation}\label{fini1}\|f\|_{H^p}\ge \left(\frac{p}{p-1}\right)^{1/p} \left(\|u\|_{h^p}^p+\|v\|_{h^p}^p\right)^{1/p}.\end{equation}

\end{theorem}
\begin{remark}
If $C_p= \left(\frac{p}{p-1}\right)^{1/p}$, then $2^{-1/p}L_p\le C_p$ if $ p\ge 4$ and $C_p\le 2^{-1/p}R_p$,  if $p_1\le p\le 4$, so inequalities \eqref{fini} and \eqref{fini1} improve the inequalities \eqref{motiv} and \eqref{motiv1}, if $p_1\le p\le 4$, and if $p\ge 4$, respectively. Here $p_1\approx 2.42484$ is the only solution of the equation
$$\left(\frac{p}{p-1}\right)^{1/p}= 2^{-1/p}\frac{1}{\sin\frac{\pi}{2p}},\ \ \ p\ge 2.$$
\end{remark}

The proofs are presented in sections~2 and ~3 and ~4.

\section{Proof of Theorem~\ref{cp}}

From Proposition~\ref{verbi} we obtain
$$L^p_p\int_{\mathbf{T}} |\Re f(re^{it})|^{p}dt\le \int_{\mathbf{T}} |f(re^{it})|^{p}dt \le  R^p_p\int_{\mathbf{T}} |\Re f(re^{it})|^{p}dt.$$ So
\[\begin{split}L^{p}_p\int_0^1r dr\int_{\mathbf{T}} |\Re f(re^{it})|^{p}dt&\le \int_0^1\int_{\mathbf{T}} r|f(re^{it})|^{p}dt \\&\le  R^p_p\int_0^1r dr\int_{\mathbf{T}} |\Re f(re^{it})|^{p}dt.\end{split}\]
Thus we have the following theorem
\begin{theorem}\label{hed}
Let  $p> 2$ and  $f=\Re f +i \Im f\in b^p(\mathbf{U})$. If $\abs{\arg{f(0)}-\frac{\pi}{2}}\ge \frac{\pi }{2p}$, or $f(0)=0$, then  we have
$$L_p\|\Re f\|_{b^p}\le \|f\|_{b^p}\le R_p\|\Re f\|_{b^p}$$ and $$L_p\|\Im f\|_{b^p}\le \|f\|_{b^p}\le R_p\|\Im f\|_{b^p}.$$

\end{theorem}

From now on we will use the shorthand notation $$\int_{\mathbf{T}} f:=\frac{1}{2\pi}\int_{0}^{2\pi}f(e^{it}) dt$$ and
$$\int_{\mathbf{U}}f:=\frac{1}{\pi}\int_{\mathbf{U}}f(z) dxdy, \ \ \ z=x+iy.$$
Thus for  $p>2$ we have
\[\begin{split}(\int_{\mathbf{U}} |\Re f|^{p})^{1/p}&\le \frac{1}{L_p}(\int_{\mathbf{U}}|f|^p)^{1/p}\\&\le \frac{1}{L_p}(\int_{\mathbf{T}}|f|^{p/2})^{2/p}\\&\le\frac{R_{p/2}}{L_p}(\int_{\mathbf{T}}|\Re f|^{p/2})^{2/p}\\&= M_p(\int_{\mathbf{T}}|\Re f|^{p/2})^{2/p},
\end{split}\]
 where
$$M_p=\left\{
        \begin{array}{ll}
          \frac{\cos{\frac{\pi}{2p}}}{\cos\frac{\pi}{p}}, & \hbox{if $2<p\le 4$;} \\
          \frac{\cos{\frac{\pi}{2p}}}{\sin\frac{\pi}{p}}, & \hbox{if $p\ge 4$.} \\
        \end{array}
      \right.
$$
This finishes the proof of Theorem~\ref{cp}.
\section{Proof of Theorem~\ref{c4}}
We assume that $f(z) = g(z)+\overline{h(z)}$, where $h(0)=0$, and $g$ and $h$ are holomorphic on the unit disk.
\begin{lemma}
The function $|a|^2+|b|^2$ is log-subharmonic, provided that $a$ and $b$ are analytic.
\end{lemma}
\begin{proof}
We need to show that $f(z)=\log(|a|^2+|b|^2)$ is subharmonic. By calculation we find $$
f_z= \frac{a'\bar a+b'\bar b}{|a|^2+|b|^2}$$ and so $$f_{z\bar z}= \frac{a'\bar a'+b'\bar b'}{|a|^2+|b|^2}-\frac{a'\bar a+b'\bar b}{|a|^2+|b|^2}
\frac{a\bar a'+b\bar b'}{|a|^2+|b|^2}$$ $$=\frac{(|a'|^2+|b'|^2)(|a|^2+|b|^2)-|\bar aa'+\bar bb'|^2}{(|a|^2+|b|^2)^2}$$ which is clearly positive.

\end{proof}

Now from the isoperimetric inequality for log-subharmonic functions (e.g. \cite[Lemma~2.2]{ka}), we have
\begin{lemma}\label{ipl}
For every positive number $p$ and analytic functions $a$ and $b$ defined on the unit disk $U$  we have that
$$\int_{\mathbf{U}}(|a|^2+|b|^2)^{2p}\le \left(\int_{\mathbf{T}}(|a|^2+|b|^2)^{p}\right)^2.$$
\end{lemma}

Further let
$$L=\int_{\mathbf{U}} (|g + \bar h|^2)^4=\int_{\mathbf{U}} (|g|^2 +  |h|^2+2\Re (gh))^4.$$
Then

\[\begin{split}
L&
 = \sum_{k=0}^4 \binom{4}{k}\int_{\mathbf{U}} (|g|^2+|h|^2)^k (2\Re(gh))^{4-k}\\&\le \sum_{k=0}^4\binom{4}{k}\int_{\mathbf{U}} ((|g|^2+|h|^2)^4)^{k/4} (\int_{\mathbf{U}} |2\Re(gh)|^4)^{(4-k)/4}.
\end{split}\]
Let $E_4 = 
            {\cos\frac{\pi}{8}}
       $.
From Lemma~\ref{ipl} and Theorem~\ref{hed} we have
\[\begin{split}L  &  \le \sum_{k=0}^4 \binom{4}{k} E_4^{4-k} (\int_{\mathbf{U}} (|g|^2+|h|^2)^4)^{k/4} (\int_{\mathbf{U}} (2|gh|)^4)^{(4-k)/4}
\\&\le \sum_{k=0}^4 \binom{4}{k} E_4^{4-k} (\int_{\mathbf{T}} (|g|^2+|h|^2)^{4/2})^{2k/4} (\int_{\mathbf{T}} (2|gh|)^{4/2})^{2(4-k)/4}.\end{split}\]
Further we have
\[\begin{split}
X:&=\int_{\mathbf{T}} |g + \bar h|^4\\&=\int_{\mathbf{T}} (|g|^2 +  |h|^2+2\Re (gh))^2\\&=\int_{\mathbf{T}} (|g|^2 +  |h|^2)^2+4(|g|^2 +  |h|^2)\Re (gh)^2+4(\Re (gh))^2
\\&=\int_{\mathbf{T}} (|g|^2 +  |h|^2)^2+4(|g|^2 +  |h|^2)\Re (gh)+2|gh|^2.\end{split}\]
Let $$A= \left(\int_{\mathbf{T}} (|g|^2 +  |h|^2)^2\right)^{1/2}$$ and
$$B = \left(\int_{\mathbf{T}} 4 |g|^2|h|^2\right)^{1/2}.$$
Then
\[\begin{split}\left|\int_{\mathbf{T}} (|g|^2 +  |h|^2)\Re (gh)\right|&\le  |(\int_{\mathbf{T}} ((|g|^2 +  |h|^2))^2)^{1/2} (\int_{\mathbf{T}} |gh|^2/2)^{1/2}
\\&=A B\cdot \frac{\sqrt{2}}{4}.\end{split}\]
Further we have $$X\ge A^2+\frac{B^2}{2}-\sqrt{2}AB=(A-\frac{\sqrt{2}}{2}B)^2.$$
Furthermore $$A^2- B^2=\int_{\mathbf{T}} (|g|^2-|h|^2)^2\ge 0$$ and thus
\begin{equation}\label{x1}X\ge (\frac{2-\sqrt{2}}{2})^2 B^2\end{equation}
and similarly
\begin{equation}\label{x2}
X\ge \left(\frac{2-\sqrt{2}}{2}\right)^2 A^2.
\end{equation}
Hence
\[\begin{split}\int_{\mathbf{U}} (|g + \bar h|)^8&\le \sum_{k=0}^4 \binom{4}{k}E_4^{4-k}\left(\frac{2}{2-\sqrt{2}}\right)^4 \left(\int_{\mathbf{T}} |g + \bar h|^4\right)^2
\\& = \left(\frac{2(1+E_4)}{2-\sqrt{2}}\right)^4 \left(\int_{\mathbf{T}} |g + \bar h|^4\right)^2.\end{split}\]
So $$\|g+\bar h\|_{b^8}\le \sqrt{\frac{2(1+E_4)}{2-\sqrt{2}}}\|g+\bar h\|_{h^4},$$ where $E_4=\cos\left[\frac{\pi }{8}\right]=\frac{\sqrt{2+\sqrt{2}}}{2}$.
Finally we have that $$\|g+\bar h\|_{b^8}\le \frac{1}{2\sin\frac{\pi}{16}}\|g+\bar h\|_{h^4}.$$ Here $\frac{1}{2\sin\frac{\pi}{16}}\approx 2.56292$.
This finishes the proof of Theorem~\ref{c4}.

\section{Proof of Theorem~\ref{newt}}

We use the following form of Green formula

\begin{equation}\label{green}r \int_0^{2\pi} \frac{\partial G(re^{it})}{\partial r} dt = \int_{|z|\le r} \Delta G(z) dx dy.\end{equation}

Let $p>2$ and let $ f = u + i v$ be an analytic function.
Let $q=\frac{p}{p-1}$ and $\epsilon>0$ and define
$$F_\epsilon(z)=|q\epsilon+|f(z)|^2|^{p/2}$$

$$U_\epsilon(z)=|\epsilon+u(z)^2|^{p/2}$$

$$V_\epsilon(z)=|\epsilon+v(z)^2|^{p/2}$$

Then by direct calculation we obtain
\begin{equation}\label{d1}\Delta F_\epsilon = p(q\epsilon+|f|^2)^{p/2-2} (2q\epsilon + p |f|^2)|f'|^2,\end{equation}

\begin{equation}\label{d2}\Delta U_\epsilon=p(u^2+\epsilon)^{p/2-2} |f'|^2(\epsilon+(p-1) u^2),\end{equation} and

\begin{equation}\label{d3}\Delta V_\epsilon=p(v^2+\epsilon)^{p/2-2} |f'|^2(\epsilon+(p-1) v^2).\end{equation}

If $p=4$, then $$\Delta U_\epsilon+\Delta V_\epsilon=\frac{3}{4}\Delta F_\epsilon .$$

%$$\Delta u^4=p (p-1)|f'|^2 u^{p-2}$$ and $$\Delta v^4=p (p-1) |f'|^2 v^{p-2}.$$

\begin{lemma}\label{new} Let $f$ be a holomorphic function defined on the unit disk $\mathbf{U}$.
For $p>4$ and  $z\in\mathbf{U}$ and $\epsilon>0$ we have
$$\Delta (U_\epsilon + V_\epsilon)\le \left(\frac{p-1}{p}\right)\Delta F_\epsilon  .$$
For $1\le p\le 4$ and  $z\in\mathbf{U}$ and $\epsilon>0$ we have
$$\Delta (U_\epsilon + V_\epsilon)\ge \left(\frac{p-1}{p}\right)\Delta F_\epsilon  .$$
\end{lemma}
\begin{proof} Let $f=u+iv=r e^{is}$ and define $$Q(s)=\frac{\Delta (U_\epsilon + V_\epsilon)}{\Delta F_\epsilon }.$$ Then from \eqref{d1}, \eqref{d2} and \eqref{d3} we have
$$Q(s)=\frac{ \left(\epsilon+r^2 \cos^2 s\right)^{-2+\frac{p}{2}} \left(\epsilon+(p-1) r^2 \cos^2 s\right)}{ \left(2 \epsilon+(p-1) r^2\right)\left(\frac{\epsilon p}{-1+p}+r^2\right)^{-2+\frac{p}{2}}}$$
$$
+\frac{  \left(\epsilon+r^2 \sin^2 s\right)^{-2+\frac{p}{2}} \left(\epsilon+(p-1) r^2 \sin^2 s\right)}{ \left(2 \epsilon+(p-1) r^2\right)\left(\frac{\epsilon p}{-1+p}+r^2\right)^{-2+\frac{p}{2}}}.$$
We should prove that

$$Q(s)\left\{
        \begin{array}{ll}
          \ge & \hbox{if $p<4$;} \\

\le & \hbox{if $p\ge 4$.}
        \end{array}
      \right.$$
 First of all $$Q'(s)=\frac{(-2+p) r^2 \left(\frac{\epsilon p}{p-1}+r^2\right)^{2-\frac{p}{2}} \cos s \sin s }{2 \epsilon+(p-1) r^2}T$$ where

\[\begin{split}T&=\left(\epsilon+r^2 \cos^2 s\right)^{-3+\frac{p}{2}} \left(-3 \epsilon+(1-p) r^2 \cos^2 s\right)\\&+\left(\epsilon+r^2 \sin^2 s\right)^{-3+\frac{p}{2}} \left(3 \epsilon+(p-1) r^2 \sin^2 s\right).\end{split}\]
Then $T=0$, if and only if $$L=S^{\frac{1}{2} (6-p)}=R:=\frac{3 \epsilon+(p-1) r^2 \sin^2 s}{3 \epsilon+(p-1) r^2 \cos^2 s},$$
where $$S=\frac{\epsilon+r^2 \sin^2 s}{\epsilon+r^2 \cos^2 s}.$$
If $p=6$, then $\cos^2 s=\sin^2 s$. If $p>6$ then we also have  $\cos^2 s=\sin^2 s$, because if for example $\cos^2 s>\sin^2 s $, then $L>1>R$. Similarly $\cos^2 s<\sin^2 s $ implies $L<1<R$. If $4<p<6$ and $\cos^2 s>\sin^2 s $ , then $$R-S=\frac{\epsilon (4-p) r^2 \cos(2 s)}{\left(\epsilon+r^2 \cos^2 s\right) \left(3 \epsilon+(p-1) r^2 \cos^2 s\right)}<0.$$ Thus

$$S<R<1.$$
Since $0<\frac{6-p}{2}<1$, it follows that

$$S<R^{\frac{6-p}{2}}<1.$$ So $L\neq R$. If $4<p<6$ and $\cos^2 s<\sin^2 s $, then

$$S>R>1.$$ So $$S>R^{\frac{6-p}{2}}.$$ Thus $L\neq R$.

If $p<4$, then $\cos^2 s>\sin^2 s $ implies that $$1>S>R>R^{\frac{6-p}{2}}.$$ Finally
if $p<4$ and $\cos^2 s>\sin^2 s $. Then $$1<S<R<R^{\frac{6-p}{2}}.$$
We proved that the only stationary points of $Q$ are $$s_j =  \frac{\pi j }{4},\ \ \ j=0,\dots, 7. $$ So wee need to show that $Q(s_j)\le 1$ for $p\le 4$ and $Q(s_j)\ge 1$ for $p\ge 4$.

Let $s=0$. Show that
$$Q(0)=\frac{ \left(\epsilon^{-1+\frac{p}{2}} p+p \left(\epsilon+r^2\right)^{-2+\frac{p}{2}} \left(\epsilon+(p-1) r^2\right)\right)}{p \left(2 \epsilon+(p-1) r^2\right)\left(\frac{\epsilon p}{p-1}+r^2\right)^{\frac{p}{2}-2}}\left\{
                              \begin{array}{ll}
                                \ge 1, & \hbox{if $p\le 4$;} \\
                                \le 1, & \hbox{if $p\ge 4$.}
                              \end{array}
                            \right.$$

Let $\epsilon= t r^2$, $t>0$. Then $$Q(0)= \frac{ \left( \left(1+\frac{p t}{{p-1}}\right)\right)^{2-\frac{p}{2}} \left(t^{-1+\frac{p}{2}}+(1+t)^{-2+\frac{p}{2}} ({p-1}+t)\right)}{{p-1}+2 t}.$$ By using convexity of the function $k(x)=x^{p/2-2}$, for
 $p<4$ we have
$$\frac{t^{-1+\frac{p}{2}}}{{p-1}+2 t}+\frac{(1+t)^{-2+\frac{p}{2}} ({p-1}+t)}{{p-1}+2 t}>\left(\frac{{p-1}+p t+2 t^2}{{p-1}+2 t}\right)^{-2+\frac{p}{2}}.$$ Further
$$\left(\frac{{p-1}+p t+2 t^2}{{p-1}+2 t}\right)^{-2+\frac{p}{2}}>\left(\frac{{p-1}+p t}{{p-1}}\right)^{-2+\frac{p}{2}}.$$
So $Q(0)\ge 1$.

For $4\le p$
 we need to show that

$$\frac{t^{-1+\frac{p}{2}}}{{p-1}+2 t}+\frac{(1+t)^{-2+\frac{p}{2}} ({p-1}+t)}{{p-1}+2 t}<\left(\frac{{p-1}+p t}{{p-1}}\right)^{-2+\frac{p}{2}},$$
i.e.
$$\frac{t^{-1+\frac{p}{2}}}{{p-1}+2 t}+\frac{(1+t)^{-2+\frac{p}{2}} ({p-1}+t)}{{p-1}+2 t}<\left(1+t+\frac{t}{{p-1}}\right)^{-2+\frac{p}{2}}.$$
The last inequality is equivalent with the trivial inequalities
$$\frac{t(t^{-2+\frac{p}{2}} -(t+1)^{-2+\frac{p}{2}})}{{p-1}+2 t}<0<\left(1+t+\frac{t}{{p-1}}\right)^{-2+\frac{p}{2}}-(1+t)^{-2+\frac{p}{2}}.$$
For $s=\pi/4$, and $\epsilon=r^2 t$, $$Q(s)=2^{2-\frac{p}{2}} \left(\frac{{p-1}+p t}{(p-1) (1+2 t)}\right)^{2-\frac{p}{2}}.$$
So $$Q(s)\left\{
           \begin{array}{ll}
             \ge 1, & \hbox{if $p\le 4$;} \\
             \le 1, & \hbox{if $p\ge 4$.}
           \end{array}
         \right.$$

The other cases can be treated in a similar way.
\end{proof}

\begin{proof}[Proof of Theorem~\ref{newt}]
Assume that $p\le 4$. By using the Green formula \eqref{green}, and Lemma~\ref{new} we obtain
\[\begin{split}r\int_0^{2\pi}\frac{\partial}{\partial r}  F_\epsilon dt&=\int_0^r\int_0^{2\pi}\rho\Delta F_\epsilon  d\rho  dt.
\\&
 \le  \frac{p}{p-1}\int_0^r\int_0^{2\pi}\rho (\Delta U_\epsilon +\Delta V_\epsilon)d\rho dt\\&=  r \frac{p}{p-1}\int_0^{2\pi}\frac{\partial}{\partial r}  (U_\epsilon+V_\epsilon) dt.
 \end{split}\]
Dividing by $r$ and integrating in $[0,1]$ with respect to $r$ we obtain
\[\begin{split}\int_0^{2\pi}  [F_\epsilon(e^{it})&-F_\epsilon(0)]  dt \\&\le   \frac{p}{p-1}\bigg(\int_0^{2\pi} [U_\epsilon(e^{it})-U_\epsilon(0)] dt+\int_0^{2\pi} [V_\epsilon(e^{it})-V_\epsilon(0)] \bigg)dt.\end{split}\]
Letting $\epsilon\to 0$ we obtain
$$\int_0^{2\pi}  |f(e^{it})|^p dt \le   \frac{p}{p-1}\left(\int_0^{2\pi} |u(e^{it})|^p dt+\int_0^{2\pi} |v(e^{it})|^p \right)dt.$$
Similarly we prove the related inequality for $p\ge 4$.
\end{proof}


\begin{thebibliography}{1}
\bibitem{axl}
\textsc{S.  Axler, P. Bourdon,  W. Ramey}, {\it Harmonic function
theory}, Springer Verlag New York 1992.

\bibitem{eb}
\textsc{E. Bajrami}, \emph{Improvement of isoperimetric type inequality for harmonic functions in the case $p=4$}, to appear in Indagationes Mathematicae, http://dx.doi.org/10.1016/j.indag.2016.10.003.




\bibitem{tc}

\textsc{T. Carleman,} {\it Zur Theorie der Minimalfl\"achen.} (German) Math.
Z. 9 (1921),  no. 1-2, 154--160.
\bibitem{spw}
\textsc{Sh. Chen, S. Ponnusamy, and X. Wang,} \emph{The isoperimetric type and Fejer-Riesz type inequalities for pluriharmonic mappings,}  Scientia Sinica Mathematica (Chinese Version), 44(2)(2014), 127--138.

\bibitem{duren}
\textsc{P. L. Duren,} \emph{Theory of $H^p$ spaces.} Pure and Applied Mathematics, Vol. 38 Academic Press, New York-London 1970 xii+258 pp.
\bibitem{hw}
\textsc{F. Hang, X. Wang, X. Yan,} {\it Sharp integral inequalities for
harmonic functions.}  Comm. Pure Appl. Math. 61 (2008), no. 1,
54--95.

\bibitem{hh}
\textsc{H. Hedenmalm,} Bloch functions, asymptotic variance, and geometric zero packing,  arXiv:1602.03358.

\bibitem{verb1}
\textsc{B. Hollenbeck, I. E. Verbitsky,} \emph{Best constants for the Riesz projection.} J. Funct. Anal. 175 (2000), no. 2, 370--392.
%\bibitem{lars}
%\textsc{L. H\"ormander,} \emph{An introduction to complex analysis in several variables.} D. Van Nostrand Co., Inc., Princeton, N.J.-Toronto, Ont.-London 1966 x+208 pp.



\bibitem{ka}
\textsc{D.  Kalaj,}  \emph{Isoperimetric inequality for the polydisk.} Ann. Mat. Pura Appl. (4) 190 (2011), no. 2, 355--369
\bibitem{ka1}
\textsc{D. Kalaj, R. Me\v strovi\' c,} \emph{Isoperimetric type inequalities for harmonic functions.} J. Math. Anal. Appl. 373 (2011), no. 2, 439--448.

\bibitem{mp}
\textsc{M. Mateljevi\'c, M.  Pavlovi\'c,} {\it New proofs of the
isoperimetric inequality and some generalizations.} J. Math. Anal.
Appl. 98 (1984), no. 1, 25--30.

\bibitem{pik}
\textsc{S. K. Pichorides,}
\emph{On the best values of the constants in the theorems of M. Riesz, Zygmund and Kolmogorov.}
Collection of articles honoring the completion by Antoni Zygmund of 50 years of scientific activity, II.
Studia Math. 44 (1972), 165--179.


\bibitem{ks}
\textsc{K. Strebel:} {\it Quadratic differentials. Ergebnisse der Mathematik
und ihrer Grenzgebiete (3)}, 5. Springer-Verlag, Berlin, 1984.
xii+184 pp.

\bibitem{ver}
\textsc{I. E. Verbitsky,}
\emph{Estimate of the norm of a function in a Hardy space in terms of the norms of its
real and imaginary parts.}
Linear operators.
Mat. Issled. No. 54 (1980), 16-20, 164--165.
\end{thebibliography}
\end{document}